\documentclass[12pt,leqno]{article}
\usepackage[ansinew]{inputenc}
\usepackage{amsmath,amsthm,amsfonts,amssymb,amscd,epsfig,graphicx}
\usepackage{graphicx}
\usepackage{epsfig}
\usepackage{psfrag}
\usepackage{xcolor}
\usepackage{cancel}
\usepackage{enumerate}
\usepackage[active]{srcltx}

\theoremstyle{plain}
\newtheorem{thm}{\sc {\bf Theorem}}[section]
\newtheorem{lem}[thm]{\sc {\bf Lemma}}
\newtheorem{coro}[thm]{\sc {\bf Corollary}}
\newtheorem{prop}[thm]{\sc {\bf Proposition}}

\theoremstyle{definition}
\newtheorem{rem}[thm]{\sc {\bf Remark}}

\newtheorem{defi}[thm]{\sc {\bf Definition}}

\newcommand{\trace}{{\rm trace}}
\newcommand{\dive}{{\rm div}}

\newcommand{\R}{{\mathbb R}}

\font\sc=cmcsc10
\begin{document}

\title {On the stability of constant higher order mean curvature hypersurfaces in a  Riemannian manifold}
\author{Maria Fernanda Elbert and Barbara Nelli}
\date{}
\maketitle

 \let\thefootnote\relax\footnote{ 
 \indent \indent The authors were partially supported by INdAM-GNSAGA and CAPES-Brasil - Print UFRJ - 2668/2018/88881.311616/2018-01.}

%
%
%
%
%

\begin{abstract}

{\em We propose a notion of stability for constant $k$-mean curvature hypersurfaces in a general Riemannian manifold and we give some applications. When the ambient manifold is a Space Form, our notion coincides with the known one, given by means of the variational problem. Our approach led us to work with two different stability operators and we are able to relate stability to the study of the respective first eigenvalues. Moreover, we prove that embedded rotational spheres with constant $k$-mean curvature in ${\mathbb H}^n\times{\mathbb R}$ or in ${\mathbb S}^n\times{\mathbb R}$ are not stable. }

\end{abstract}

\section*{Introduction}

Let $x:M^{n}\longrightarrow \bar{M}^{n+1}$ be an isometric immersion
of an orientable connected Riemannian $n$-manifold into an oriented Riemannian n+1-manifold
 and let
$A:T_{p}M\longrightarrow T_{p}M$ be the linear operator associated to 
the second fundamental form of $x$. Denote by $\kappa_{1}, \kappa_{2},..., \kappa_{n}$ its 
eigenvalues, namely the principal curvatures of $x$. The 
elementary symmetric functions of $\kappa_{1}, \kappa_{2},..., \kappa_{n},$ say
$$\begin{array}{ll}
H_{0}=1,&\\[8pt]
H_{k}={\displaystyle \sum_{i_{1}<...<i_{k}}\kappa_{i_{1}}... \kappa_{i_{k}}},\hspace{1cm}&
(1\leq k\leq n),\\
H_{k} = 0,&(k>n)
\end{array}$$
are known as the (non normalized)   {\em r-mean curvatures of $x$}.

A hypersurface with  constant ${k}$-mean curvature is called a {\em $H_{k}$-hypersurface.} Notice that $H_1$ is the non normalized mean curvature and then a $H_{1}$-hypersurface is a constant mean curvature hypersurface.

 Many examples of $H_k$-hypersurfaces have been constructed over the years. 
For example, W. Y. Hsiang \cite{Hs}, M. L. Leite \cite{Le} and O. Palmas \cite{P} described 
$H_k$-hypersurfaces invariant by rotations in Space Forms (M. L. Leite for $k=2$). The first author and R. Sa  Earp \cite{ElSa}, the authors and W. Santos \cite{ENS}, R. F. De Lima, F. Manfio, J. P. Do Santos \cite{FMS}, the second author, G. Pipoli and G. Russo \cite{NPR}, described $H_k$-hypersurfaces, including $H_k=0,$ invariant by rotations in some product spaces.
Properties of $H_k$-hypersufaces either in Space Forms or in product spaces were studied, for example, in \cite{AIR, BrEi, ElNe, ENS, Ro}. Looking closely at examples and properties of $H_{k}$-hypersurfaces one can see that they share many features with constant mean curvature hypersurfaces.

 We recall that  a  $H_{k}$-hypersurface in a Space Form is a critical point for a modified area functional.  This was proved in  R. Reilly 's pioneer article \cite{Re} in  the Eclidean case, while J. L. Barbosa and G. Colares \cite{BC}  proved it in a general Space Form. Notice that the modified area functional coincides with the area functional when $k=1.$

 We point out that  a similar variational
characterization of $H_k$-hypersurfaces in a general Riemannian manifold is not known. It is worth noticing that the very definition of the modified area functional in \cite{BC}  involves the constant sectional curvature of the ambient space.

 Inspired by \cite{BGM}, where the authors address stability for prescribed mean curvature hypersurfaces, we  propose a notion of  stability for $H_{k}$-hypersurfaces in a general Riemannian manifold.   A similar approach was used to deal with the stability of marginally outer trapped surfaces, the so called MOTS (\cite{AEM, AMS, AM, GS}).

Here and in these latter cases, the definition of stability is given in terms of the associated stability operator T, instead of that given by means of a variational problem. The operator T  is not self-adjoint, but we are also able to introduce a symmetrized stability operator. 
With both definitions at handy, we are able  to address stability questions of $H_k$-hypersurface beyond  the framework of Space Forms.
For example,  we prove the following result analogous to the Space Forms case.

\

{\em 
Embedded rotational $H_{k}$-spheres in ${\mathbb H}^n\times{\mathbb R}$ or in ${\mathbb S}^n\times{\mathbb R}$ are not (weakly) stable. 
}

\

The paper is organized as follows. In Section 1, we recall the basic concepts related to the higher order mean curvature of a hypersurface. In Section 2, after recalling the notion of stability for $H_k$-hypersurfaces in Space Forms, we define our notion of stability related to a non-self adjoint {\it stability} operator and explore the corresponding eigenvalue theory. In Section 3, we obtain the {\it symmetrized stability operator}, propose an associated {\it symmetric stability} definition and highlight the relation between both stability definitions. Finally, in Section 4, we give some applications.

 \section{Preliminaries}
\label{prelim}

 Let $x:M^{n}\longrightarrow \bar{M}^{n+1}$ be an isometric immersion  
of an orientable connected Riemannian $n$-manifold into an oriented Riemannian $n+1$-manifold
 and let
$A$ be the second fundamental form of $x$. 
Let $e_{1},e_{2},...,e_{n}$ be orthonormal eigenvectors of $A$ 
corresponding, respectively, to the eigenvalues $\kappa_{1}, \kappa_{2},..., \kappa_{n}$.

The study of the higher order mean curvatures is related to the classical Newton transformations $P_{r}$ defined inductively by
$$\begin{array}{l}P_{0}= I,\\
P_{r}=H_{r}I-AP_{r-1},\\
\end{array}$$
here $I$ is the identity matrix. Each $P_{r}$ is a self-adjoint operator
 that has the same eigenvectors of $A$.

Lemma \ref{eleerre} below establishes some conditions to be satisfied by the immersion in order to guarantee that
 $P_{r}$ is definite.

\begin{lem} 

\ 

\begin{description}
\item[(a)]  If $r=1$ and $H_{2}>0$  then, after a proper choice of the orientation of the immersion,   $P_1$ is positive definite;
\item[(b)] If $H_{r+1}>0$, $r>1$,  and if there exists a point $q$ in $M$ where all the principal 
curvatures of $x$ are positive, then $P_r$ is positive definite;
\item[(c)]If $H_{r+1}=0$ and $rank(A)>r$ then, $P_r$ is positive or negative definite.
\end{description}
\label{eleerre}
\end{lem}
\begin{proof}
The statements  (a) and  (b) are well known. The proof of (c) can be found
in \cite[Corollary (2.3)]{HL}.
\end{proof}

 For later use, it is useful to recall some notation and results about stability related to  the variational approach.

Let $\bar{M}^{n+1}$ be an oriented Riemannian 
$(n+1)$-manifold  and let $x:M^{n}\longrightarrow \bar{M}^{n+1}$ be an isometric immersion of a
compact with boundary $\partial M$ (possibly empty) connected oriented Riemannian $n$-manifold into $\bar{M}^{n+1}$. 
 By a {\em variation of $x$} we mean a differentiable
map $X:(-\varepsilon,\varepsilon)\times M\longrightarrow
\bar{M}^{n+1}$, $\varepsilon >0$, such that for each $t\in 
(-\varepsilon,\varepsilon)$ the map  $X_{t}:\{t\}\times M
\longrightarrow\bar{M}^{n+1}$   defined by $X_{t}(p)=X(t,p)$ is an immersion, $X_{0}=x$, and $X_t|_{\partial M}=X|_{\partial M}$. Set
$$
E_{t}(p)=\frac{\partial X}{\partial t}(t,p)\;\;\;
\mbox{and}\;\;\;f_{t}=\left<E_{t},N_{t}\right >,
$$
where $N_{t}$ is the unit normal vector field in $X_{t}$. E is the 
{\em variational vector field} of $X$.

\

Denote by  $(.)^{T}$ and $(.)^{N}$  the tangent
and normal component to $M,$ respectively,  and  by $\bar{\nabla}$ and
$\nabla$, the connection of $\bar{M}$ and the connection of $M$ 
in the metric induced by $\phi_{0}$, respectively.  Also denote by  $A(t)$  the second fundamental form of $X_{t}$. 
The following results hold.

\begin{lem} {{\rm (\cite[Lemma (3.1)]{E})}}
$$ A'(t)=Hessf+f\bar{R}_{N}+fA^{2}.$$
Here, $\bar{R}_{N}(Y)=\big(\bar{R}(Y,N)N)^T$, where
$\bar{R}$ is the curvature tensor of $\bar{M}^{n+1}$.
\label{var1}
\end{lem}

\begin{rem}
 In [Ro] (cf. Formula (3.4)) one can find a proof of Lemma \ref{var1} when $\bar{M}^{n+1}$ is a Space Form.

\end{rem}

\begin{prop}
\label{proposition-variation}
We have:\begin{description} 
    \item[(a)] $\trace(P_r A^2)=H_{1}H_{r+1}-(r+2)H_{r+2}$
    \item[(b)] $\frac{\partial}{\partial t}(H_{r+1})=L_{r}(f)+f\trace(P_r A^2)
+f\trace(P_{r}\bar{R}_{N}),$ where $L_r(f)=tr(P_r(Hess f))$.
\end{description}
\label{sr}
\end{prop}

 The proof of (a) is a standard computation and the proof of (b) can be seen in {{\rm (\cite[Proposition (3.2)]{E})}}

 \

We notice that $L_{r}$ written as
 
 \begin{equation}
L_{r}(f)=\dive(P_{r}\nabla f)-\left<\trace(\nabla P_r),\nabla f\right>,
\label{CY}
\end{equation}
where $\nabla f$ is the gradient of $f$, $\nabla P_r$ is the covariant derivative of the tensor $P_r$, $\trace(\nabla P_r)$ is the trace with respect to the metric of $M$  and ${\rm div (.)}$ is the divergence operator  on $M$.

\

Now we recall some properties and results concerning  an elliptic self-adjoint
linear differential operator of second order  $T.$ By elliptic, we mean that the matrix of the principal part is positive definite.

 A domain $D\subset M$ is an open connected subset with compact closure and smooth boundary.  In the case $M$ is compact without boundary, $D$ is allowed to be equal to $M$.
Denote by $C_{0}^{\infty}(D)$  the set of 
smooth functions which are zero on $\partial D$. 

\ 

 We recall that the first eigenvalue 
$\lambda_{1}(T,D)$ of $T$ in $D$ is defined 
as the smallest $\lambda$ that satisfies
\begin{equation}
T(g)+\lambda g=0,\;\;\;\;g\in C_{0}^{\infty}(D) 
\label{tres}
\end{equation}
where $g$ is a non-identically zero function. A  non-identically zero function $g$ in 
$C_{0}^{\infty}(D)$ that satisfies (\ref {tres}) 
for $\lambda=\lambda_{1}(T,D)$ is called a first eigenfunction of $T$ in $D$.

\

\

The following lemma is a standard result. 

\begin{lem} 
 $$
\lambda_{1}(T, D)=\inf\left\{\frac{-\int_{D}fT(f)dM}
 {\int_{D}f^{2}dM};f\in H_0^{1}(D)
,f\equiv\hspace*{-.38cm} /\; 0\right\},
$$
where $H_0^{1}(D)$ is the first Sobolev Space over $D$.
\label{minmax}
\end{lem}

\

\begin{prop}  {{\rm (\cite[Proposition (3.13)]{E}, \cite[Theorem 1]{FCS})}}Suppose that $M^n$ is complete and noncompact and that $T$ is a second order self-adjoint elliptic operator.
The following statements are equivalent:
\begin{description}
\item[(a)] $\lambda_{1}(T, D)\geq 0$ for every domain $D\subset M$;
\item[(b)] $\lambda_{1}(T, D)>0$ for every domain $D\subset M$;
\item[(c)] There exists a positive smooth function $f$ on $M$ satisfying the 
equation \linebreak $T f=0$.
\end{description}
\label{FCschoen}
\end{prop}

\section{Stability}
\label{stab-sf}

 We denote by  $\bar{M}^{n+1}(c)$ the  complete, simply connected Riemannian manifold with constant sectional curvature $c$, i.e. a Space Form and we consider an immersion $x:M^{n}\longrightarrow \bar{M}^{n+1}(c)$. We recall that, when the  ambient manifold is a Space Form,  $H_{r+1}$-hypersurfaces are critical points of the variational problem of minimizing the
integral $${\cal A}_{r}=\int_{M} F_{r}(H_{1},...,H_{r})dM,$$ for 
compactly supported variations (cf. \cite{BC, Re, Ro}). ${\cal A}_{r}$ is known as {\em $r$-area} of the hypersurface and the functions $F_{r}$ are
defined inductively by
$$\begin{array}{l}
F_{0}=1,\\
F_{1}=H_{1},\\
F_{r}=H_{r}+\frac{c(n-r+1)}{r-1}F_{r-2},\;\;\;2\leq r\leq n-1.
\end{array}$$

Associated to the second variation formula of the variational  problem described above  is
the second order \mbox{differential} operator 
$$L_{r}+\trace(A^2P_r)+c(n-r)H_{r},$$
where $L_{r} (f)=\trace(P_{r}(Hess(f)) $  was defined in  Proposition \ref{proposition-variation}.

\

When dealing with the classical case of constant mean curvature, $r=0$, we have $P_0=I$ and then $L_0=\trace (Hess f)$ is the Laplacian, which is elliptic. Lemma \ref{eleerre} gives conditions on the immersion that guarantee the definiteness of $P_r$ and since we want to deal with elliptic operators we need to restrict ourselves to the case in which $P_r$ is positive definite. It will be clear in a while that the case in which $P_r$ is negative definite can be addressed with slight modifications (see Remark \ref{remark-ellipticity}).

\begin{defi}
We say that $H_{r+1}$-hypersurface is {\it positive definite} if $P_r$ is positive definite.
\end{defi}

\begin{defi}{(stability operator in Space Forms)} The {\it stability operator} of a positive definite $H_{r+1}$-hypersurface in a Space Form $M^{n+1}(c)$ is defined by

 $$T_{r}=L_{r}+\trace(A^2P_r)+c(n-r)H_{r}$$

\end{defi}

\

\begin{defi}{(stability in Space Forms)} Let $x:M^n\longrightarrow \bar{M}^{n+1}(c)$  be a positive definite $H_{r+1}$-hypersurface. The immersion $x$ is {\it stable}  if $-\int_{M}fT_r(f)\;dM\geq 0$ for all variations of $x$.  If $M$ is 
noncompact, we say that $x$ is stable if for every domain $D\subset M$, $x|_D$ is stable. 
 \label{operador}
\end{defi}

\ 

\begin{rem}
 If we want to address the case in which $P_r$ is negative definite, the stability operator should be defined by $T_r^{\small Neg}:=-T_r=-\left(L_{r}+\trace(A^2P_r)+c(n-r)S_{r}\right)$.

  Looking at  the related  variational problem with the functional ${\cal A}_r$, we have:
$$
\left(\frac{d^2}{dt^2}{\cal A}_r\right)|_{t=0}=-\int_{M}f(L_{r}(f)+\trace(A^2P_r)f+c(n-r)H_{r}f)\; dM.
$$  
  We see that when $P_r$ is positive definite, then $\left(\frac{d^2}{dt^2}{\cal A}_r\right)|_{t=0}=-\int_{M}fT_r(f)\;dM$, and stability gives the usual notion
of minimum. When $P_r$ is negative definite, we have  $\left(\frac{d^2}{dt^2}{\cal A}_r\right)|_{t=0}=\int_{M}fT_r^{\small Neg}(f)\;dM$ and stability gives the maximum; this is equivalent to look for the minimum of the new variational problem
$-{\cal A}_r$. See (\cite{ADCE}), where a similar discussion takes place. 
\label{remark-ellipticity}
\end{rem}

\

Proposition \ref{fc} below gives a characterization of stability of a positive definite $H_{r+1}$-hypersurface in a Space Form.

\begin{prop}
\label{fc}
Let $x:M^n\longrightarrow \bar{M}^{n+1}(c)$ be a complete  positive definite $H_{r+1}$-hypersurface in a Space Form. Then  $M$ is stable if and only if there exists a positive function $u \in C^{2}(M^n)$ such that $T_r(u)\leq 0$. 
\label{estab}
\end{prop}

\begin{proof}

Let us assume that $M$ is stable. If $M$ is compact without boundary, Lemma \ref{minmax} implies that $\lambda_1(T_r,M)\geq 0$ and we can take $u \in C^{\infty}(M)$ to be the positive corresponding eigenfunction. Then, (\ref{tres}) yields $T_r(u)\leq 0$. If $M$ is complete and noncompact, stability and Lemma \ref{minmax} imply $\lambda_1(T_r,D)\geq 0$. Then Proposition \ref{FCschoen} gives the existence of $u$.

Now, let us assume the existence of a positive function $u \in C^{2}(M)$ such that $T_r(u)\leq 0$. 

  By hypothesis, there exists $u\in C^\infty (M)$, $u>0$, such that $L_{r}(u) + qu\leq 0$, where $q=\trace(A^2P_r)+c(n-r)S_{r}.$ Given $f\in C_c^\infty (M)$, the function $\varphi=\frac{f}{u}$ lies in $C_c^\infty (M)$. Integration by parts gives
 \begin{eqnarray*}
 &- &\int_M (L_r(f) +qf)f
 \;dM =\\
 &-&\int_M  \left(\varphi u\left< \nabla u,P_{r}(\nabla \varphi)\right>+ \varphi u^2 L_{r}(\varphi)+\varphi u\left<\nabla \varphi, P_r\nabla u\right>+\varphi^2 u L_r(u) +\varphi^2 u^2q\right)\;dM\geq\\
 &-&\int _{M} \left(\varphi u\left< \nabla u,P_{r}(\nabla \varphi)\right>+ \varphi u^2 L_{r}(\varphi)+\varphi u\left<\nabla \varphi, P_r\nabla u\right>\right)\;dM=
 \\
 &-&\int_{M}  \left(2\varphi u\left< \nabla u,P_{r}(\nabla \varphi)\right>+ \varphi u^2 L_{r}(\varphi) \right)\;dM = \int_{M} u^2\left< \nabla \varphi,P_{r}(\nabla \varphi)\right>\;dM\geq 0.
 \end{eqnarray*}
 Then $M$ is stable.
\end{proof}
 
 \
 
 In Space Forms, the linearized operator of the $(r+1)$-mean curvature equation (see Proposition \ref{sr}), $T_r$, is associated with a variational problem involving the $r$-area  of the immersion. In a general Riemannian manifold, since there is no known similar variational problem, we use the characterization of $T_r$ given by Proposition \ref{estab} in order to define stability.

\begin{defi}{(stability operator in a general Riemannian manifold)}
\label{operator-defi}
The {\it stability operator} of a positive definite $H_{r+1}$-hypersurface of $\bar{M}^{n+1}$ by
$$T_{r}=L_{r}+\trace(A^2P_r)+\trace(P_{r}\bar{R}_{N}).$$
\label{gstab}
\end{defi}

\begin{defi}{(stability in a general Riemannian manifold)}\label{defi-stability}
Let $x:M^n\longrightarrow \bar{M}^{n+1}$ be a positive definite $H_{r+1}$-hypersurface. We say that $M$ is {\em stable} if there exists a positive function $u\in C^\infty (M)$  such that $T_r u\leq 0$.  
\end{defi}

 By Proposition \ref{estab}, Definition \ref{defi-stability} agrees with the standard definition of stability for $H_{r+1}$-hypersurfaces in Space Forms and, for $r=0$, coincides with the classical notion of stability for constant mean curvature hypersurfaces. 
 Finally, Definition \ref{defi-stability} is also consistent
with the notion of stability in MOTS theory (see \cite[Definition 2]{AMS},\cite[Definition 3.1]{AEM}).

\

 When the ambient space is a Space Form, it can be shown that $\trace(\nabla P_r)=0$ in (\ref{CY}) and then 
\begin{equation}
L_{r}(f)=\dive(P_{r}\nabla f),
\label{div}
\end{equation}
 (see \cite[Theorem 4.1]{Ro}). In  this  case, we can use Stokes Theorem and the 
self-adjointness of $P_{r}$ to see that $L_{r}$ is self-adjoint.
 In a general Riemannian manifold  $\bar{M}^{n+1}$, equation (\ref{div}) may not hold and the stability operator $T_r$ (Definition \ref{gstab}) is not in general self-adjoint. Hence we are not allowed to apply the properties of the eigenvalues and the solutions of equation $T_r(g)+\lambda g=0$, which were described at the end of  Section {\ref{prelim}}. Nevertheless, it is possible to develop an eigenvalue theory and to obtain an interpretation of 
 stability in terms of a special eigenvalue of $T_r$, which is defined by

\begin{equation}
{\lambda}_p(T_r,D)=\sup\left\{ \inf_{p\in D}\left\{-\frac{T_r(u)(x)}{u(x)}  \right\} {\big \vert} u\in C^\infty_0 (D), \; u>0  \right\},
\label{evans}
\end{equation}
and is called the {\it principal eigenvalue} of $T_r$( 
See  \cite[(3.1)]{LiCoWa} with ${\mathcal L}_2=-T_r$).


 The following Proposition is a direct application of \cite[Theorem 3.1]{LiCoWa}.
 
 \begin{prop} 
 \label{principal-eigenvalue-prop}
 Let $x:M^n\longrightarrow \bar{M}^{n+1}$ be a positive definite $H_{r+1}$-hypersurface and  $D\subset M$. Then
 
 \begin{itemize}
\item[{\rm (a)}] the eigenvalue ${\lambda}_p(T_r,D)$ is real and there exists a positive eigenfunction  $g\in C^{\infty}_0(D).$ Furthermore, if ${\lambda}(T_r,D)\in\mathbb{C}$ is any other eigenvalue, we have $Re({\lambda}(T_r,D))\geq {\lambda}_p(T_r,D);$
 \item[{\rm (b)}] the eigenvalue ${\lambda}_p(T_r,D)$ is simple.
 \end{itemize}
 \label{principal}
 \end{prop}

The next proposition gives a characterization of stability in terms of ${\lambda}_p(T_r,D)$.

\begin{prop}
 Let $x:M^n\longrightarrow \bar{M}^{n+1}$ be a positive definite $H_{r+1}$-hypersurface.  Then $M$ is  stable if and only if ${\lambda}_p(T_r,D)\geq 0$ for all bounded $D\subset M$.
 \label{stability-eigenvalue}
\end{prop}
\begin{proof} The proof is analogous to the proof of \cite[Lemma 2]{AMS}.

If ${\lambda}_p(T_r,D)\geq 0$, let $\phi$ be the corresponding positive eigenfunction given by Proposition \ref{principal}. Then $T_r(\phi)\leq 0$ and $M$ is stable.

For the converse, let $T^*_r$ be the formal adjoint of $T_r$ (with respect to the standard $L^2$ inner product in $D$). $T^*_r$ also has ${\lambda}_p(T_r,D)$ as the principal eigenvalue and a positive corresponding eigenfunction $\phi^*$. Since $M$ is stable, there exists a positive function $u\in C^\infty (D)$  such that $T_r u\leq 0$ with $T_r u\not\equiv 0$.

Then
$$
{\lambda}_p(T_r,D)\int_D \phi^* u\;dM=-\int_D T^*_r(\phi^*)u\;dM=-\int_D \phi^*T_r(u)\;dM\geq 0.
$$
\end{proof}

 \section{Symmetric stability }
 
 \



In this section, we introduce the notion of symmetric stability (see Definition \ref{symm-stability}) for $H_k$-hypersurfaces, which is stated in terms of a self-adjoint operator. We are going to prove that stability (see Definition \ref{gstab}) implies symmetric stability.

 The next Proposition is inspired by \cite[ Lemma 4.8]{BGM} (see also the proof of \cite[ Theorem 2.1]{GS}).

\begin{prop}
\label{proposition-sym}
Let $(\Sigma,\langle,\rangle)$ be a Riemannian manifold and consider the following  differential operator on $\Sigma$ 
$$T=\dive(\Phi\nabla \cdot)+\langle \sqrt{\Phi}X,\sqrt{\Phi}\nabla\cdot\rangle+q,$$

where $X$ is a vector field in $\Sigma$ and $q$ is a continuous function on $\Sigma$. Let us suppose that $\Phi$ is a positive definite self-adjoint operator  and that there exists $u\in C^2(\Sigma)$, $u>0$, with $Tu\leq 0$. Then, the operator
$$
T^S=div(\Phi\nabla \cdot)+Q,\hspace{1cm} Q=q-\frac{1}{2}div(\Phi X)-\frac{|\sqrt{\Phi}X|^2}{4}.
$$
satisfies $-\int_\Sigma f{T^S}f\;d\Sigma\geq 0$ for all  compactly supported function $f\in C^2(\Sigma)$.
\label{stab}
\end{prop}

\begin{proof}
Let $u>0$ be such that $Tu\leq 0$. The inequality $Tu\leq 0$ gives
$$
div(\Phi\nabla u)+\frac{u}{4} |2\sqrt{\Phi}\nabla \ln u+\sqrt{\Phi}X|^2  - \frac{u}{4}|\sqrt{\Phi}X|^2-u|\sqrt{\Phi}\nabla \ln u|^2 +qu\leq 0.
$$

If we write $u=e^g$, the previous inequality can be rewritten as
$$
\dive(\Phi\nabla g)+\frac{1}{4} |2\sqrt{\Phi}\nabla g+\sqrt{\Phi}X|^2  - \frac{1}{4}|\sqrt{\Phi}X|^2 +q\leq 0.
$$

Adding and subtracting $\frac{1}{2} \dive(\Phi X)$ we obtain
$$
\dive(\Phi(\nabla g +X/2))+ |\sqrt{\Phi}(\nabla g+X/2)|^2  - \frac{1}{4}|\sqrt{\Phi}X|^2 -\frac{1}{2} \dive(\Phi X) +q\leq 0.
$$

Now, we define $Y:=\nabla g +X/2$ and $Q:=q- \frac{1}{4}|\sqrt{\Phi}X|^2 -\frac{1}{2} \dive(\Phi X)$ and the previous inequality writes as

$$
|\sqrt{\Phi} Y|^2+Q\leq -\dive(\Phi Y).
$$

 Then for  $f\in C^2(\Sigma)$ with compact support,  we obtain

$$
\begin{array}{rcl}
f^2|\sqrt{\Phi} Y|^2+f^2Q&\leq& -f^2\dive(\Phi Y)\\[8pt]
                          &=& -\dive(f^2(\Phi Y))+2f\langle\nabla f,\Phi Y\rangle\\[8pt]
                          &=& -\dive(f^2(\Phi Y))+2f\langle\sqrt{\Phi}\nabla f,\sqrt{\Phi} Y\rangle\\[8pt]
                          &\leq&  -\dive(f^2(\Phi Y))+2|f|                           \; |\sqrt{\Phi}\nabla f|\;|\sqrt{\Phi} Y       |\\[8pt]
                          &\leq& -\dive(f^2(\Phi Y))+                                         |\sqrt{\Phi}\nabla f|^2 +f^2|\sqrt{\Phi} Y       |^2  .
                                             
 \end{array}
$$

By integrating both sides of the last inequality we obtain

$$
\int_\Sigma |\sqrt{\Phi}\nabla f|^2-f^2Q\;d\Sigma\geq 0,  \;\;\;\; \forall \mbox{  compactly supported function } f\in C^2(\Sigma).
$$
or equivalently
$$
-\int_\Sigma f\dive(\Phi\nabla f)+f^2Q \;d\Sigma=-\int_\Sigma f{ T^S}f\;d\Sigma\geq 0, \;\;\;\;\forall \mbox{  compactly supported function } f\in C^2(\Sigma).
$$
\end{proof}

\begin{coro}Let $x:M^n\longrightarrow \bar{M}^{n+1}$ be a positive definite $H_{r+1}$-hypersurface and suppose that $M$ is stable. Let
$$
T^S_r=\dive(P_r\nabla \cdot)+\trace(P_r A^2)+\trace(P_{r}\bar{R}_{N})-\frac{1}{2}\dive(P_r X)-\frac{|\sqrt{P_r}X|^2}{4},
$$

where $X=- P_r^{-1}\trace(\nabla P_r)$. Then $T^S_r$ is self-adjoint and satisfy
$-\int_M fT^S_rf\;dM\geq 0$ for all for all  compactly supported function $f\in C^2(M)$.
\label{symm}
\end{coro}

\begin{proof}

We recall that the stability operator is given by
$$
T_{r}=L_{r}+\trace(P_r A^2)+\trace(P_{r}\bar{R}_{N}),
$$

which, in view of equation \eqref{CY} and of the fact that $P_r$ is a self-adjoint operator, can be rewritten as

$$
\begin{array}{rcl}
T_{r}&=&\dive(P_{r}\nabla \cdot)-\left<\trace(\nabla P_r),\nabla \cdot\right>+\trace(P_r A^2)+\trace(P_{r}\bar{R}_{N})\\[8pt]
     &=& \dive(P_{r}\nabla \cdot)-\left<P_r P_r^{-1}\trace(\nabla P_r),\nabla \cdot\right>+\trace(P_r A^2)+\trace(P_{r}\bar{R}_{N})\\[8pt]
     &=& \dive(P_{r}\nabla \cdot)-\left<\sqrt{P_r} P_r^{-1}\trace(\nabla P_r),\sqrt{P_r}\nabla \cdot\right>+\trace(P_r A^2)+\trace(P_{r}\bar{R}_{N}).
     \end{array}
$$

Now we define $X=- P_r^{-1}\trace(\nabla P_r)$ and $q=\trace(P_r A^2)+\trace(P_{r}\bar{R}_{N})$ and the result follows by Proposition \ref{stab}.
\end{proof}

Then we can give the following Definition.

\begin{defi}
\label{sso}
 The {\it symmetrized stability operator} of a positive definite $H_{r+1}$-hypersurface in  $M^{n+1}$ is defined by $$T^S_r=\dive(P_r\nabla \cdot)+\trace(P_r A^2)+\trace(P_{r}\bar{R}_{N})-\frac{1}{2}\dive(P_r X)-\frac{|\sqrt{P_r}X|^2}{4}.$$

\

where $X=- P_r^{-1}\trace(\nabla P_r)$.
\end{defi}

\

Also, we define:

\begin{defi}{(Symmetric stability in a general Riemannian manifold)}\label{symm-stability}
Let $x:M^n\longrightarrow \bar{M}^{n+1}$ be a positive definite $H_{r+1}$-hypersurface.The immersion $x$ is  {\em symmetric-stable}  if $-\int_{M}fT^S_r(f)\;dM\geq 0$ for all variations of $x$.  If $M$ is noncompact, we say that $x$ is symmetric stable if for every domain $D\subset M$, $x|_D$ is symmetric stable. 
 \label{symstable}
\end{defi}

 Notice that,  Definitions \ref{operador}, \ref{defi-stability} and \ref{symstable} are consistent with the notion of {\bf strong stability} in the classical literature about constant mean curvature hypersurfaces.  The other classical notion is that of {\bf weak stability} that corresponds to the further assumption   $\int_M f\;dM= 0$.

\

\begin{prop}\label{eigenvalues}
Let $x:M^n\longrightarrow \bar{M}^{n+1}$ be a positive definite $H_{r+1}$-hypersurface. 
Then, for any domain $D\subset M,$ one has:
$$\lambda_p (T_r, D)\leq \lambda_1 (T^S_r, D). $$
\label{compare-eigenvalues}
\end{prop}

\begin{proof} The proof is inspired in the proof of [Theorem 2.1,\cite{GS}]. 

Let $f_0 \in C_0^\infty(D)$ be the first eigenfunction associated to the eigenvalue $\lambda(T_r, D),$ and let $\widetilde{f}=\log f_0 \in C^\infty(D)$. We have \begin{eqnarray}
 &-&\lambda_p (T_r, D)= \frac{T_rf_0}{f_0}=\frac{1}{f_0}\dive\left(P_r\nabla f_0\right)-\left<{\rm trace}(\nabla P_r),\frac{\nabla f_0}{f_0}\right>+{\rm trace}(P_r A^2+\bar{R}_{N}) \nonumber \\
&=& {\rm \dive}\left(P_r\frac{\nabla f_0}{f_0}\right)-\left<P_r\nabla f_0,\nabla\left(\frac{1}{f_0}\right)\right>+\left<P_r\left(-P_r^{-1}{\rm trace}(\nabla P_r)\right),\nabla\widetilde{f}\right>+{\rm trace}(P_r A^2+\bar{R}_{N}) \nonumber \\
&=& {\rm div}\left(P_r\left(\nabla\widetilde{f}+\frac{X}{2}\right)\right)+\left\vert{\sqrt{P_r}\left(\nabla\widetilde{f}+\frac{X}{2}\right)}\right\vert^2+{\rm trace}(P_r A^2+\bar{R}_{N})-\frac{\vert\sqrt{P_r}X\vert^2}{4}-\frac{{\rm div}\left(P_rX\right)}{2}. \nonumber
\end{eqnarray}
Now, we set $Y_r=\nabla\widetilde{f}+\frac{X}{2}$ and $Q={\rm trace}(P_r A^2+\bar{R}_{N})-\frac{\vert\sqrt{P_r}X\vert^2}{4}-\frac{\dive\left(P_rX\right)}{2}$. Thus, if $f \in C_0^\infty(D)$, $f\not\equiv 0$, we obtain

\begin{eqnarray}
&-&\lambda_p (T_r, D) \int_D f^2 dM= \int_D f^2\dive\left(P_rY_r\right)+f^2\vert\sqrt{P_r}Y_r\vert^2+f^2Q\,dM \nonumber \\
&=& \int_D -\left<P_rY_r,\nabla\left(f^2\right)\right>+f^2\vert\sqrt{P_r}Y_r\vert^2+Qf^2\,dM+\int_D \dive\left(f^2P_rY_r\right)\,dM \nonumber \\
&=& \int_D -2f\left<\sqrt{P_r}\nabla f,\sqrt{P_r}Y_r\right>+f^2\vert\sqrt{P_r}Y_r\vert^2+f^2Q\,dM.\nonumber\\
&=& \int_D |\sqrt{P_r}\nabla f-f\sqrt{P_r}Y_r|^2- |\sqrt{P_r}\nabla f|^2-f^2|\sqrt{P_r}Y_r|^2+f^2\vert\sqrt{P_r}Y_r\vert^2+f^2Q\,dM.\nonumber\\
&=& \int_D |\sqrt{P_r}\nabla f-f\sqrt{P_r}Y_r|^2- |\sqrt{P_r}\nabla f|^2+f^2Q\,dM.\nonumber\\
&\geq& \int_D  -|\sqrt{P_r}\nabla f|^2+f^2Q\,dM=\int_D  f\dive(P_r\nabla f)+f^2Q\,dM.\nonumber
\end{eqnarray}

Then $$\lambda_p (T_r, D)  \int_D f^2 dM\leq -\int_D fT^S_r(f) dM$$
and the result now follows from Rayleigh quotient characterization for $\lambda (T^S_r, D)$.
\end{proof}

\
Corollary \ref{symm} and Proposition \ref{compare-eigenvalues} yield that stability implies symmetric stability. Also, by using  properties of $\lambda_1(T^S_r,D)$ we obtain the following result.

\begin{prop}
Let $x:M^n\longrightarrow \bar{M}^{n+1}$ be a complete  positive definite $H_{r+1}$-hypersurface and suppose that $M$ is stable. Then:

\begin{description}
\item [(a)]$-\int_{M} fT^S_rf\;dM\geq 0$ for all $f\in C^2_0(M)$.
\item [(b)]  $\lambda_1(T^S_r,D)\geq 0$ for every $D\subset M$.
\item [(c)] There exists a positive smooth function $f$ on $M$ satisfying the equation $T^S_r f\leq 0$.
\end{description}
\end{prop}

\

We notice that, when the ambient space is a Space Form $\bar{M}^{n+1}(c)$, then $T^S_r=T_r$ and and both notions of stability, namely stability and symmetric stability, coincide.

\section{Some applications}

In the next theorem, we give a lower bound for the real
 eigenvalue of the stability operator $T_r$. The proof is inspired by \cite[Theorem 3.2]{BJLM}.

\

\begin{thm}
\label{thm_pacelli} Let $\bar{M}^{n+1}$ a riemannian manifold with negative sectional curvature and 
let $x:M^n\longrightarrow \bar{M}^{n+1}$ be   a positive definite $H_{r+1}$-hypersurface. Assume  that  the second fundamental form of the immersion satisfies $ 0<\frac{\trace(P_r A^2)}{\trace(P_r)}(p)\leq -Sec(\bar{M})$, where $Sec(\bar{M})$ is the infimum of the sectional curvatures of $\bar{M}^{n+1}$.  Let $B_{\bar{M}}(p,R)$ be a geodesic ball centered at $p\in \bar{M}^{n+1}$ of radius $R$ and let $\Omega$ be a connected component of $x^{-1}(\overline{B_{\bar{M}}(p,R)})$. We have

 $${\lambda}_p(T_r,\Omega)\geq \frac{2}{R^2} \left\{(n-r)\inf_{\Omega}H_r- (r+1)H_{r+1}R\right\}.$$

\end{thm}

  \begin{proof}
  Let $g:B_{\bar{M}}(p,R)\to\R$ given by $g=R^2-\rho^2$, where $\rho(\cdot)=(dist(\cdot,p))$ is the distance function in $\bar{M}$ and set $f=g\circ x$. Let $e_1, e_2,\ldots,e_n$ be the eigenvectors of $P_r$ 
  and $\mu_i^r)$ the corresponding eigenvalues. We can write (see \cite[(24)]{BJLM}).
  $$
  L_r(f)=\displaystyle{\sum_{i=1}^n \mu_i^r.Hess\; g(e_i,e_i)}+ (r+1)H_{r+1}.\left<\nabla g,N\right>.
  $$

By definition (see \ref{evans}), we see that
  \begin{eqnarray*}
{\lambda}_p(T_r,\Omega)&\geq & \inf_{\Omega}\left\{-\frac{T_r(f)}{f}\right\}\\[8pt]
&=&
\inf_{\Omega}\left\{-\frac{1}{g}\left[  \displaystyle{\sum_{i=1}^n \mu_i^rHess\ g(e_i,e_i)}+ (r+1)H_{r+1}\left<\nabla g,N\right>\right]\right.\\[8pt]& &\Biggl. \Biggl.-\trace(P_r A^2)-\trace(P_{r}\bar{R}_{N})\Biggr\}
\end{eqnarray*}

By using the hypothesis on the sectional curvature, we have that $\trace(P_r A^2)+\trace(P_{r}\bar{R}_{N})\leq 0$ and we obtain

\begin{eqnarray}
\label{ineq1}
{\lambda}_p(T_r,\Omega)&\geq & 
\inf_{\Omega}\left\{-\frac{1}{g}\left[ \displaystyle{\sum_{i=1}^n \mu_i^rHess\  g(e_i,e_i)}+ (r+1)H_{r+1}\left<\nabla g,N\right>\right]\right\}.
\end{eqnarray}

 Now, we follow  the proof of \cite[,Theorem 3.2]{BJLM}.  For the sake of completeness we sketch the argument here. 
 The Hessian of $g$ is given by
 
 \begin{equation*}
 Hess\  g(e_i,e_i)=-2\langle\nabla\rho, e_i\rangle^2-2\rho Hess\rho(e_i,e_i).
 \end{equation*}
 
 Moreover, by Hessian comparison Theorem \cite[Theorem 1.1]{SY} and \cite[Lemma 5.2]{BJLM}   
 
 \begin{equation*}
 \sum_{i=1}^n\mu_i^r[\langle\nabla\rho, e_i\rangle^2+\rho Hess\rho(e_i,e_i)]\geq  \sum_{i=1}^n\mu_i^r[\Vert e_i^T\Vert^2+\rho V(\rho)
 \Vert e_i^{\perp}\Vert^2]
  \end{equation*}

 where $e_i^T=\langle\nabla\rho, e_i\rangle\nabla\rho,$ $e_i^{\perp}=e_i-e_i^T$ and $V(\rho)=-\sqrt{-Sec(\bar M)}\coth \sqrt{-Sec(\bar M)}\rho$
 
 Replacing the previous relation in \eqref{ineq1}, one has 
 
 \begin{eqnarray*}
\label{ineq2}
{\lambda}_p(T_r,\Omega)&\geq  
2{\displaystyle\inf_{\Omega}} \frac{1}{R^2-\rho^2}\left[ \displaystyle{\sum_{i=1}^n \mu_i^r[\Vert e_i^T\Vert^2+\rho V(\rho)
 \Vert e_i^{\perp}\Vert^2]-(r+1)RH_{r+1}}\right]\\&\geq
 \frac{2}{R^2}{\displaystyle\inf_{\Omega}} \left[\displaystyle{\sum_{i=1}^n \mu_i^r-(r+1)RH_{r+1}}\right]\\&=
  \frac{2}{R^2}\left[\displaystyle{(n-r){\displaystyle\inf_{\Omega}}H_r-(r+1)RH_{r+1}}\right]
\end{eqnarray*}

 where in the second inequality we use that $\rho V(\rho)\geq 1$ and  in the last equality we use that $\trace(P_r)=(n-r)H_r.$
 \end{proof}

  An immediate consequence of Theorem \ref{thm_pacelli}  is the following stability result.
 
 \begin{coro}
 Assume $x:M^n\longrightarrow \bar{M}^{n+1}$ satisfies the hypothesis of {\rm Theorem \ref{thm_pacelli}}. Moreover assume that  
 $(n-r)\inf_{\Omega}H_r- (r+1)H_{r+1}R>0.$ Then, any connected domain of $M$ contained in $x^{-1}(\overline{B_{\bar{M}}(p,R)}),$ where $B_{\bar{M}}(p,R)$ is a geodesic ball of $\bar{M}^{n+1},$ is stable.

 \end{coro}

 In the following, we get some results about the stability of graphs and rotational spheres.
\
\begin{lem}
\label{lemma-0-eigenvalue}
Let $x:M^n\longrightarrow \bar{M}^{n+1}$  be a  positive definite $H_{r+1}$-hypersurface in a manifold $ \bar{M}^{n+1}$  carrying a killing vector field ${\bf X}.$ Let ${\bf N}$ be a unit normal vector field defined on $M.$ Then $T_r\langle {\bf N}, {\bf X}\rangle=0,$ where $T_r$ is the stability operator of $M.$
\end{lem}

\begin{proof}
 As $X$ is  a Killing vector field, the associated variation is an isometry, hence
 all the hypersurfaces obtained by the variation have the same $(r+1)$-mean curvature.
Then, by Proposition  \ref{proposition-variation}, $T_r\langle {\bf N}, {\bf X}\rangle=0.$
\end{proof}

A direct consequence of Lemma \ref{lemma-0-eigenvalue} and of Definition \ref{defi-stability} is the following result.

\begin{prop}
Let  $x:M^n\longrightarrow \bar{M}^{n+1}$ be a positive definite $H_{r+1}$-hypersurface in a manifold $ \bar{M}^{n+1}$  carrying a killing vector field ${\bf X}.$ Let ${\bf N}$ be a unit normal vector field defined on $M$ and assume that $\langle {\bf N}, {\bf X}\rangle>0.$ Then $M$  is stable.
\end{prop}


\begin{prop}
Let $x:M^n\longrightarrow \Sigma^n\times{\mathbb R}$ be a  positive definite $H_{r+1}$-hypersurface   that is a graph over $\Sigma^n$.
Then $M$ is stable.
\end{prop}

\begin{proof} 
The proof is trivial because the vertical vector field $\frac{\partial}{\partial t}$ in  $\Sigma^n\times{\mathbb R}$ is a Killing vector field and $\langle N,\frac{\partial}{\partial t}\rangle$ may be choose to be positive for a graph.
\end{proof}

Let $M\subset \bar M.$ We say that an isometry $\Phi$ of $\bar M$ preserves $M$ if $\Phi (M)=M.$
 
\begin{prop}
Assume that $x:M^n\longrightarrow \bar{M}^{n+1}$ is a compact  positive definite $H_{r+1}$-hypersurface and that there exist  two independent Killing vector fields in $\bar M^{n+1}$ whose associated isometries  do not preserve $M.$
Then  $M$ is not  stable. 
\label{prop4.5}
\end{prop}

\begin{proof}By Lemma \ref{lemma-0-eigenvalue}, $0$ is an eigenvalue for $T_r$ and each Killing vector field generate an associated eigenfunction. Then $0$ has at least multiplicity two. Assume that  $M$ is a compact and stable positive definite $H_{r+1}$-hypersurface. By taking $D=M$ in Proposition \ref{principal-eigenvalue-prop}, we conclude that $0$ can not be the principal eigenvalue, since it is not simple, and that  the principal eigenvalue  must be negative. This is a contradiction to Proposition \ref{stability-eigenvalue}. 
\end{proof}

In the next  theorem  we prove that compact embedded rotational $H_{r+1}$-spheres of ${\mathbb H}^n\times{\mathbb R}$  or ${\mathbb S}^n\times{\mathbb R}$ are not  stable.  Rotational hypersurfaces in such spaces  are defined  as those which are foliated by horizontal geodesic spheres centered at an axis  $\{0\}\times\mathbb{R},$ $0\in {\mathbb H}^n$ or $0\in {\mathbb S}^n$. In \cite{ElSa,  Ron}, besides other examples, the authors construct embedded rotational $H_{r+1}$-spheres in ${\mathbb H}^n\times\mathbb{R}$ and ${\mathbb S}^n\times\mathbb{R}$. Also, they prove that any compact connected embedded $H_{r+1}$-hypersurface in these spaces is  a rotational sphere.

 \

%
%

%
%

\begin{thm}
Embedded rotational $H_{r+1}$-spheres in ${\mathbb H}^n\times{\mathbb R}$ or in ${\mathbb S}^n\times{\mathbb R}$ are not stable.
\end{thm}

\begin{proof}
We first notice that an embedded rotational $H_{r+1}$-sphere in one of these spaces is  strictly convex, hence it is positive definite (see \cite[Theorem 1 and Theorem 4]{FMS}). By construction, it  has two points on the rotational axis.
Also, we can easily see that the embedded rotational $H_{r+1}$-sphere is not preserved by  vertical translation and by at least one "horizontal" isometry.  By "horizontal" isometry we mean the isometry in ${\mathbb H}^n\times{\mathbb R}$ (respectively  ${\mathbb S}^n\times{\mathbb R}$), induced by an isometry in ${\mathbb H}^n$(respectively in ${\mathbb S}^n$). Then, Proposition \ref{prop4.5} gives the result.

 \end{proof} 
 
%

In Space Forms, spheres  are weakly  stable \cite{BDE, BC}.  In \cite{Sou}, R. Souam proves that   embedded rotational $H_{1}$-spheres in  ${\mathbb H}^2\times{\mathbb R}$ are  weakly stable and gives a value $h_0$ of the mean curvature such that   for $H_1>h_0$  rotational $H_{1}$-spheres in  ${\mathbb S}^2\times{\mathbb R}$ are weakly stable. Recentely Souam's result was extended to higher dimension by R. F. de Lima   and the authors \cite{DEN}.  Hence  a natural question arises.

\

{\em Question.}  Under which conditions embedded rotational $H_{r+1}$-spheres  in 
${\mathbb H}^n\times{\mathbb R}$ or ${\mathbb S}^n\times{\mathbb R}$  are  weakly stable?








\end{document}